\documentclass[11pt,reqno]{amsart}
\usepackage{amssymb,amsmath}
\usepackage[utf8]{inputenc}
\usepackage{amsthm}
\usepackage{amssymb}
\usepackage{amsmath,amscd}
\usepackage{latexsym}
\usepackage{graphicx,subfigure}
\usepackage[
    hscale=0.70,
    vscale=0.775,
    headheight=13pt,
    centering,
]{geometry}
\usepackage{tikz-cd}
\usepackage{xcolor}
\definecolor{verydarkblue}{rgb}{0,0,0.5}
\usepackage[
    breaklinks,
    colorlinks,
    citecolor=verydarkblue,
    linkcolor=verydarkblue,
    urlcolor=verydarkblue,
    pagebackref=true,
    hyperindex
]{hyperref}

\usepackage[capitalize]{cleveref}

\theoremstyle{plain}

\newtheorem{introtheorem}{Theorem}

\crefname{introtheorem}{Theorem}{Theorems}

\newtheorem{theorem}{Theorem}
\newtheorem{lemma}[theorem]{Lemma}
\newtheorem{proposition}[theorem]{Proposition}

\newtheorem{corollary}[theorem]{Corollary}

\theoremstyle{definition}
\newtheorem{definition}[theorem]{Definition}
\newtheorem{example}[theorem]{Example}

\theoremstyle{remark}

\newtheorem{remark}[theorem]{Remark}

\numberwithin{theorem}{section}
\numberwithin{equation}{section}

\newcommand{\Q}{\mathbb{Q}}
\newcommand{\PP}{\mathbb{P}}

\newcommand{\R}{\mathbb{R}}
\newcommand{\C}{\mathbb{C}}
\newcommand{\Z}{\mathbb{Z}}

\newcommand{\FF}{\mathbb{F}}

\newcommand{\Aff}{\mathbb{A}}
\newcommand{\aid}{\mathfrak{a}}

\newcommand{\OS}{\mathcal{O}}

\newcommand{\I}{\mathcal{I}}
\newcommand{\F}{\mathcal{F}}

\newcommand{\LL}{\mathcal{L}}
\newcommand{\M}{\mathcal{M}}

\newcommand{\X}{\mathcal{X}}
\newcommand{\Y}{\mathcal{Y}}

\newcommand{\spec}{\operatorname{Spec}}

\newcommand{\proj}{\operatorname{Proj}}

\newcommand{\ord}{\operatorname{ord}}
\newcommand{\vol}{\operatorname{vol}}

\newcommand{\DF}{\operatorname{DF}}
\newcommand{\Ding}{\operatorname{Ding}}
\newcommand{\lct}{\operatorname{lct}}
\newcommand{\glct}{\operatorname{Glct}}

\newcommand{\im}{\operatorname{Im}}
\newcommand{\aut}{\operatorname{Aut}}
\newcommand{\val}{\operatorname{Val}}

\newcommand{\Gval}{\operatorname{GVal}}
\newcommand{\gr}{\operatorname{gr}}
\newcommand{\PGL}{\operatorname{PGL}}

\title[A Note on Equivariant K-stability]{A Note on Equivariant K-stability}
\author{Ziwen Zhu}
\address{Department of Mathematics, University of Utah, Salt Lake City, UT 84112, USA}
\email{{\tt zzhu@math.utah.edu}}
\thanks{Research partially supported by NSF Grant DMS-1700769}

\begin{document}
\begin{abstract}
We define $G$-pseudovaluations on a variety with a group action $G$. By introducing $G$-pseudovaluations, we are able to give some criteria for $G$-equivariant K-stability of Fano varieties which are parallel to existing results for usual K-stability.
\end{abstract}
\maketitle
\section{Introduction}
We work over the field $\C$ of complex numbers. A $\Q$-Fano variety is a normal projective variety with klt singularities such that the anti-canonical divisor is $\Q$-Cartier and ample. 

It was conjectured that in order to test K-(poly/semi)stability of a $\Q$-Fano variety it is enough to examine equivariant test configurations with respect to a finite or connected reductive subgroup $G$ of $\aut(X)$. For the case of Fano manifolds, an analytic proof was given in \cite{DS16}. When $G$ is a torus group, for any $\Q$-Fano variety, an algebraic proof was provided in \cite{LX16} for K-semistability and in \cite{LWX18} for K-polystability. Recently after the first version of this paper had been posted online, the finite group case of the conjecture was solved in \cite{LZ20}, and then a full solution to the conjecture eventually was given by Zhuang in \cite{zhuang2020}.

The purpose of this short note however is to provide another perspective on equivariant K-stability for $\Q$-Fano varieties with an arbitrary group action. By replacing the space of valuations with a special collection of pseudovaluations in terms of the group action, we give parallel results to two existing criteria on characterizing K-stability from \cite{Fuj16} and \cite{Li17K,LL19,LX16}. Indeed, for any variety $X$, let $G\subset \aut(X)$ denote a group action on $X$. For any valuation $v$ on $X$, we define
$$
G\cdot v:=\inf_{g\in G} g\cdot v,
$$
where $g$ acts on the valuation $v$ by $g\cdot v(f)=v(f\circ g)$ for any $f\in \C(X)$. We call $G\cdot v$ a $G$-pseudovaluation and denote all $G$-pseudovaluations on $X$ by $\Gval_X$. Note that all $G$-invariant valuations, which we denote by $\val_X^G$, are contained in $\Gval_X$. For any $G$-pseudovaluation $G\cdot v$, and a nonnegative real number $x$, we can define the ideal sheaf $\aid_x(G\cdot v)$ to be 
$$
\aid_x(G\cdot v)=\bigcap_{g\in G} \aid_x(g\cdot v),
$$
where for any valuation $w$, $\aid_x(w)$ is the ideal sheaf of regular functions with vanishing order no less than $x$ with respect to $w$.
Refer to Section \ref{pre} for details about the definition of $G$-pseudovaluations.

The first theorem is about valuative criteria of equivariant K-stability parallel to the main results in \cite{Fuj16}. Let $X$ be a $\Q$-Fano variety and $G\subset \aut(X)$ a group action on $X$. For a prime divisor $F$ over $X$, let $\ord_F$ be the corresponding divisorial valuation. We define the $G$-equivariant beta invariant of $F$ to be
$$
\beta^G(F):=A_X(F)(-K_X)^n-\int_0^{+\infty}\vol_X\left(\OS_X(-K_X)\otimes\aid_x(G\cdot \ord_F)\right)\,dx.
$$

We say that $F$ is of finite orbit if the orbit of the valuation $\ord_F$ under $G$-action is finite. We say that $F$ is $G$-dreamy if $F$ is of finite orbit and moreover the graded ring
$$
\bigoplus_{k,j\geq 0} H^0\left(X,\OS_X(-kK_X)\otimes\aid_j(G\cdot \ord_F)\right)
$$
is finitely generated.

Define 
$$
\tau^G(F):=\sup\{t>0|\vol_X\left(\OS_X(-K_X)\otimes\aid_t(G\cdot \ord_F)\right)>0\}
$$
and
$$
j^G(F)=\int_0^{\tau^G(F)}\left(\vol_X(-K_X)- \vol_X\left(\OS_X(-K_X)\otimes\aid_x(G\cdot \ord_F)\right)\right)\,dx.
$$

Note that for $G$-invariant divisors over $X$, the above definitions coincide with the usual ones defined in \cite{Fuj16}.

The following theorem gives valuative criteria of K-stability in terms of $\beta^G(F)$:
\begin{introtheorem}\label{val}
Let $X$ be a $\Q$-Fano variety with $G\subset \aut(X)$ a group action on $X$.
\begin{enumerate}
\item The following are equivalent:
 	\begin{enumerate}
 		\item[(i)] $X$ is uniformly $G$-equivariantly K-stable;
 		\item[(ii)] there exists $0<\delta<1$, such that $\beta^G(F)\geq \delta j^G(F)$ for any finite-orbit prime divisor $F$ over $X$;
 		\item [(iii)] there exists $0<\delta<1$, such that $\beta^G(F)\geq \delta j^G(F)$ for any $G$-dreamy prime divisor $F$ over $X$.
 	\end{enumerate}
 \item The following are equivalent:
 	\begin{enumerate}
 		\item[(i)] $X$ is $G$-equivariantly K-semistable;
 		\item[(ii)] $\beta^G(F)\geq 0$ for any finite-orbit prime divisor $F$ over $X$;
 		\item [(iii)] $\beta^G(F)\geq 0$ for any $G$-dreamy prime divisor $F$ over $X$.
 	\end{enumerate}
 \item The following are equivalent:
 	\begin{enumerate}
 		\item[(i)] $X$ is $G$-equivariantly K-stable;
 		\item[(ii)] $\beta^G(F)> 0$ for any $G$-dreamy prime divisor $F$ 						over $X$.
 	\end{enumerate}
\end{enumerate}
\end{introtheorem}
\begin{remark}
When $G$ is connected, we know that every finite-orbit prime divisor is $G$-invariant. Therefore, it is enough to only consider beta invariants of $G$-invariant divisors in Theorem \ref{val} in this case. A similar result is shown in \cite{golota2019delta} in terms of the stability threshold.
\end{remark}
\begin{remark}
When $G$ is finite, every prime divisor over $X$ is of finite orbit. Moreover, for $G$-invariant K-semistability, it is enough to consider beta invariants of $G$-invariant divisors over $X$. Indeed, based on the valuative criteria in Theorem \ref{val}, Liu and the author showed further in \cite{LZ20} that it is in fact enough to check all $G$-equivariant special test configurations for $G$-equivariant K-semistability. The klt central fiber of a $G$-equivariant special test configuration of $X$ is in particular reduced and irreducible, and induces a $G$-invariant divisorial valuation on $X$. The argument in \cite{Fuj16} therefore shows that it is enough to consider $G$-invariant divisors when $G$ is finite.
In general, when $G$ is not finite, all the prime divisors induced by weakly $G$-special test configurations (see Section \ref{eqKs} for definition) are still of finite orbit. Therefore, we are not losing any information in terms of test configurations and K-stability by focusing only on divisors of finite orbit.
\end{remark}

We can also characterize equivariant K-stability in terms of equivariant normalized volume of $G$-pseudovaluations. Normalized volume of $G$-pseudovaluations can be defined similarly as the normalized volume of usual valuations in \cite{Li18} and we will use the same notation. See Section \ref{pre} for more details.

Let $X$ be a $\Q$-Fano variety with $G$-action, denote by $Y=C(X,-K_X)$ the cone over $X$ and $o\in Y$ the vertex of the cone. Suppose $\pi:~Z=Bl_oY\to Y$ is the blow-up of $Y$ at $o$. Let $E$ be the exceptional divisor of the blow-up. Denote the divisorial valuation $\ord_E$ by $v_0$. Note that there is a natural $G$-action induced on the cone $Y$ and the blow-up $Z$. Since $v_0$ is a $G$-invariant divisorial valuation, we know that $v_0\in \val^G_{Y,o}\subset \Gval_{Y,o}$, where $\val^G_{Y,o}$ and $\Gval_{Y,o}$ refer to $G$-invariant valuations and $G$-pseudovaluations with center to be $o$ respectively. (See Section \ref{pre} for the definition of the center of a $G$-pseudovaluation.)

Under the above notation, we have the following characterization of  $G$-equivariant K-semistability compared to the results in \cite{Li17K,LL19,LX16}:
\begin{introtheorem}\label{norm}
$X$ is $G$-equivariantly K-semistable iff the normalized volume function $\widehat{\vol}_{Y,o}$ is minimized at $v_0$ among all finite-orbit $G$-pseudovaluations on $Y$ centered at $o$.
\end{introtheorem}
\begin{remark}\label{unimin}
Uniqueness of the minimizer of $\widehat{\vol}_{Y,o}$ among all valuations on $Y$ centered at $o$ was originally conjectured by Chi Li in \cite{Li18}. Assuming the conjecture, we know that if $Y$ admits a $G$-action, then the minimizer is necessarily $G$-invariant. As it is well known, this would immediately imply the equivalence between $G$-equivariant K-semistability and usual K-semistability by a similar argument as in the proof of Theorem E in \cite{LX16}. In particular, it would follow that it is enough to consider only $G$-invariant divisors and $G$-invariant valuations to check K-semistability.
\end{remark}

Recently after the paper is posted online, there are two relevant results concerning Remark \ref{unimin}. Xu and Zhuang confirmed Li's conjecture on the uniqueness of the minimizer of normalized volume  in \cite{XZ20}. Zhuang also showed in \cite[Theorem 1.2]{zhuang2020}, using a different approach from \cite{XZ20}, that it is enough to check invariant divisors for K-semistability, which is a key step in his proof of the equivalence between equivariant K-stability and usual K-stability.

\subsection*{Acknowledgements}
The author would like to thank his advisor Tommaso de Fernex for proposing the question that motivates the project and providing insightful thoughts throughout the project. He would also like to thank Harold Blum, Kento Fujita, Yuchen Liu, Ivan Cheltsov and Kewei Zhang for effective discussions. Suggestions and comments from the anonymous referees of the paper are also greatly appreciated.

\section{Pseudovaluations, normalized volumes and Equivariant K-stability}\label{pre}
We include in this section relevant equivariant version of notions about valuations and K-stability for reader's convenience.
\subsection{Valuations and pseudovaluations}
We recall first the definition of the center of a valuation.
A valuation $v$ on $\C(X)^\ast$ defines a valuation ring $\OS_v\subset \C(X)$ by
$$
\OS_v=\{f\in \C(X)^\ast|v(f)\geq 0 \}\cup \{0\}.
$$
We say $v$ has a center on $X$ if we have a map $\spec \OS_v\to X$ induced by the following diagram:
\begin{center}
\begin{tikzcd}
\spec \C(X) \arrow[d]\arrow[r] & X \arrow[d] \\
\spec \OS_v \arrow [ur,dotted]\arrow[r] & \spec \C.
\end{tikzcd}
\end{center}
In this case, we call $v$ a valuation on $X$. Note that if we assume $X$ is proper, then the map $\spec \OS_v\to X$ always exists and is unique. Let $\xi$ be the image of the unique closed point of $\spec \OS_v$ in $X$. Then we call the schematic point $\xi$ the center of $v$ on $X$ and denote it by $c_X(v)$. For a closed point $x\in X$, we denote by $\val_{X,x}$ all $\R$-valued valuations of $\C(X)$ centered at $x$.

For a variety $X$ with a group action $G$, we define $G$-pseudovaluations as follows.
\begin{definition}
Let $G$ be a group action on $X$ and $v$ a valuation on $X$. Define
$$
G\cdot v:=\inf_{g\in G} g\cdot v,
$$
where $g\cdot v$ is the valuation given by $g\cdot v(f)=v(f\circ g)$ for any $f\in \C(X)$. We call $G\cdot v$ a $G$-pseudovaluation and denote all $G$-pseudovaluations on $X$ by $\Gval_X$. The \emph{center} of $G\cdot v$ is defined to be the union of the centers of $g\cdot v$ for all $g\in G$. We say $G\cdot v$ (or simply~$v$ if the group action $G$ is clear from the context) is \emph{of finite orbit} if the orbit of $v$ under $G$-action is finite.
\end{definition}

\begin{remark}
Note that by \cite[Lemma 1.7]{BHJ17}, if $v_1$ and $v_2$ are two valuations of finite orbit such that the two induced $G$-pseudovaluations $G\cdot v_1$ and $G\cdot v_2$ are the same, then $v_1$ and $v_2$ are in the same orbit under $G$-action. For counterexamples when one of the valuation is not of finite orbit, see Example \ref{csgo}.
\end{remark}

\begin{remark}
In general, the center of a $G$-pseudovaluation $G\cdot v$ is simply a union of schematic points. It might not be closed nor irreducible. For a finite-orbit $G\cdot v$, the center of $G\cdot v$ is a finite union of $c_X(g\cdot v)$. If $x\in X$ is a closed $G$-invariant point, and the center of $g\cdot v$ is $x$ for all $g\in G$, we will simply say that $G\cdot v$ has center to be $x$ or is centered at $x$.
\end{remark}
\begin{remark}
In general, $G$-pseudovaluations are not valuations because they do not satisfy the product property. Indeed, for any $f,g\in \C(X)$, we only have
$$
G\cdot v(fg)\geq G\cdot v(f)+G\cdot v(g).
$$
If $U\subset X$ is an affine open set containing all the centers of the valuations  $g\cdot v$, then $G\cdot v$ induces a pseudovaluation on $\OS_X(U)$ in the sense of \cite{arcs}. When $G\cdot v$ is of finite orbit, we can always find such $U$. Note that pseudovaluations on an affine variety do not extend to its function field due to the lack of product property. In general there is not a clear way to define pseudovaluations on a projective variety.
\end{remark}

For a valuation $v$ on $X$ and a nonnegative real number $x$, the ideal sheaf $\aid_x(v)\subset \OS_X$ is defined as follows. For $U\subset X$ an open affine subset of $X$, if $U$ contains the center of $v$, then define
$$
\aid_x(v)(U)=\{f\in \OS_X(U)^\ast|v(f)\geq x\}\cup\{0\}.
$$
If $U$ does not contain the center of $v$, we set $\aid_x(v)(U)=\OS_X(U)$. For a $G$-pseudovaluation $G\cdot v$, and $x$ a nonnegative real number, we define the ideal sheaf $\aid_x(G\cdot v)$ to be
$$
\aid_x(G\cdot v)=\bigcap_{g\in G} \aid_x(g\cdot v).
$$

\subsection{Equivariant normalized volume}
Let $x$ be a $G$-invariant point on $X$. Denote by $\Gval_{X,x}$ all $G$-pseudovaluations centered at $x$. We can define the normalized volume $\widehat{\vol}$ on the $\Gval_{X,x}$ almost in the same way as normalized volume of usual valuations. First of all, for any $G$-pseudovaluation $G\cdot v$, we define the volume 
$$
\vol(G\cdot v)=\lim_{\lambda\to \infty}\frac{\dim_{\C}\OS_{X,x}/\aid_\lambda(G\cdot v)}{\lambda^n/n!}.
$$
Note that $A_X(g\cdot v)=A_X(v)$ for any $g\in G$, so we define the log discrepancy of $G\cdot v$ to be $A_X(v)$. Then the normalized volume of $G\cdot v$ is defined as
$$
\widehat{\vol}(G\cdot v)=A_X(v)^n\vol(G\cdot v).
$$

\subsection{Equivariant K-stability}\label{eqKs}
We first give the definition of equivariant test configuration.
\begin{definition}
Let $(X,L)$ be a polarized variety. A (semi-)test configuration $(\X,\LL)$ of $(X,L)$ with exponent $r$ consists of the following data:
\begin{enumerate}
\item a proper flat family $\pi:\X\to \Aff^1$,
\item an equivariant $\C^\ast$-action on $\pi:\X\to \Aff^1$, where $\C^\ast$ acts on $\Aff^1$ by multiplication in the standard way, and
\item a $\C^\ast$-equivariant line bundle $\LL$ on $\X$ which is $\pi$-relatively (semi-)ample,
\end{enumerate}
such that $(\X,\LL)|_{\pi^{-1}(\Aff^1\backslash \{0\})}$ is $\C^\ast$-equivariantly isomorphic to $(X \times (\Aff^1\backslash \{0\}),L^{\otimes r}_{\Aff^1\backslash \{0\}})$, where $L_{\Aff^1\backslash \{0\}}$ is the pull back of $L$ from $X$ to $X \times (\Aff^1\backslash \{0\})$.

In addition, let $G\subset \aut(X,L)$ be a group action on $(X,L)$. We say $(\X,\LL)$ is a $G$-equivariant test configuration if $G$ can be extended to an action on $(\X,\LL)$ such that it commutes with the $\C^\ast$-action on $(\X,\LL)$.

$(\X,\LL)$ is called a \emph{trivial} test configuration if $(\X,\LL)$ is $\C^\ast$-equivariantly isomorphic to $(X \times \Aff^1,L^{\otimes r}_{\Aff^1})$, where the $\C^\ast$-action on $X \times \Aff^1$ is trivial on the first component.
\end{definition}

In the rest of the paper, we will focus on $\Q$-Fano varieties with the polarization to be~$-K_X$ and a group action $G\subset \aut(X)$. By replacing~$-K_X$ with a sufficiently divisible multiple of itself, we may assume~$-K_X$ is already Cartier.
The definition of Donaldson-Futaki invariant for an equivariant test configuration of a $\Q$-Fano variety is the same as the usual one. We include a definition using intersection formula here which will come up in later computation.
\begin{definition}
Let $X$ be a $\Q$-Fano variety of dimension $n$. For any rational number $r$ such that $rK_X$ is Cartier, let $(\X,\LL)$ be a normal semi-test configuration of $(X,-rK_X)$. We can compactify the test configuration into a flat family $(\bar{\X},\bar{\LL})$ over $\PP^1$, such that over $\PP^1\backslash \{0\}$, the family $(\bar{\X},\bar{\LL})$ is $\C^\ast$-equivariantly isomorphic to $X\times \PP^1\backslash \{0\}$ with trivial $\C^\ast$-action on the fibers. Then we can define the Donaldson-Futaki invariant of $(\X,\LL)$ to be
\begin{equation}\label{int}
\DF(\X,\LL):=\frac{1}{(n+1)(-K_X)^n}\left(\frac{n}{r^{n+1}}\bar{\LL}^{n+1}+\frac{n+1}{r^n}\bar{\LL}^n\cdot K_{\bar{\X}/\PP^1}\right)
\end{equation}
\end{definition}

We also include the definition of $J^{\textrm{NA}}(\X,\LL)$ following \cite{Fuj16}, which can be viewed as the norm of $(\X,\LL)$ since $J^{\textrm{NA}}(\X,\LL)=0$ if and only if $(\X,\LL)$ is a trivial test configuration. 
Let
\begin{center}
\begin{tikzcd}
& \Y \arrow[dl,"p"']
	\arrow[dr,"q"] & \\
X\times \PP^1 \arrow [rr,dotted] & & \bar{\X}.
\end{tikzcd}
\end{center}
be a common resolution of $X\times \PP^1$ and $\bar{\X}$. We set
$$
\lambda_{\max}(\X,\LL):=\frac{p^\ast (-K_{X\times\PP^1/\PP^1})^n\cdot q^\ast \bar{\LL}}{(-K_X)^n},
$$
and define
$$
J^{\textrm{NA}}(\X,\LL):=\lambda_{\max}(\X,\LL)-\frac{\bar{\LL}^{n+1}}{(n+1)(-rK_X)^n}
$$
\begin{definition}\label{def:k-stab}
Let $X$ be a $\Q$-Fano variety with $G\subset \aut(X)$ a group action on $X$. We have the following three definitions of K-stability:
\begin{enumerate}
\item $(X,-K_X)$ is said to be $G$-equivariantly K-semistable if the Donaldson-Futaki invariant is nonnegative for all $G$-equivariant normal test configurations.
\item $(X,-K_X)$ is said to be $G$-equivariantly K-stable if the Donaldson-Futaki invariant is positive for all nontrivial $G$-equivariant normal test configurations. 
\item $(X,-K_X)$ is said to be uniformly $G$-equivariantly K-stable if there exists $0<\delta<1$ such that $\DF(\X,\LL)\geq \delta J^{\textrm{NA}}(\X,\LL)$ for all $G$-equivariant normal test configurations.
\end{enumerate} 
\end{definition}

\subsection{Equivariant MMP and weakly $G$-special test configuration}
In this section, we would like to produce a collection of equivariant test configurations that plays the same role as special test configurations for K-stability following the argument in \cite{LX14}. Indeed, we would like to run similar Minimal Model Program (MMP) steps $G$-equivariantly. For readers' convenience, we collect some useful facts about $G$-MMP as follows.

For a projective variety $X$ and a group action $G\subset \aut(X)$, we can run $G$-equivariant MMP if the action of $G$ on the N\'eron-Severi group factors through a finite group (\cite[cf. Example 2.21]{KM98}). Indeed, the equivariant Contraction Theorem involves contraction of the orbit of an extremal ray under $G$-action, which is possible since the orbit is finite in this case.

When $G$ is an algebraic group, let $G_0$ be the connected component of $G$ containing the identity element. Then $G/G_0$ is finite and the action of $G$ on the N\'eron-Severi group factors through $G/G_0$. Therefore, we can always run $G$-MMP.
For an arbitrary group with possibly infinitely many connected components, we have the following lemma which allows us to run $G$-MMP on $G$-equivariant test configurations. The key fact we use here is that a Fano variety is a Mori dream space and the Mori cone is polyhedral.

\begin{lemma}\label{GMMPtc}
Let $G<\aut(X)$ be a group action on a Fano variety $X$, and $(\X,\LL)$ be a $G$-equivariant test configuration of $(X,-K_X)$. Then the $G$-action on the N\'eron-Severi group $N^1(\X)$ factors through a finite group.
\end{lemma}
\begin{proof}
By the construction of test configurations through flag ideals in \cite{Oda13a} and \cite{BHJ17}, we know that there exists a semi-test configuration $(\Y,\M)$ such that $p:\Y\to X\times \Aff^1$ is the normalized blow-up along the ideal sheaf 
$$
\I=I_r+I_{r-1}t+\cdots+I_1t^{r-1}+(t^r)\subset \OS_{X\times \Aff^1}.
$$
with
$$
\M=p^\ast (-K_{X\times \Aff^1/\Aff^1})-E,
$$
where $\OS_{\Y}(-E)=p^{-1}\I\cdot \OS_{\Y}$. For the map $\pi:\Y\to \Aff^1$, we know that $\M$ is relative semi-ample, and 
$$
\phi:\Y\to \X=\proj \bigoplus_{m\geq 0}\pi_\ast \M^m,
$$
with $\phi^\ast\LL=\M$.
Note that since $(\X,\LL)$ is $G$-equivariant, $\I$ is a $G$-invariant ideal with all $I_i$'s also G-invariant. Therefore, there is an induced $G$-action on $(\Y,\M)$. Next we show that the induced action of $G$ on $N^1(\Y)$ factors through a finite group. 

Indeed, for any divisor $D$ on $\Y$, the difference $D-p^\ast p_\ast D$ is  supported in the central fiber $\Y_0$. Therefore, we see that $N^1(\Y)=p^\ast N^1(X\times \Aff^1)\oplus N'$, where $N'$ is spanned by numerical equivalence class of non-movable divisors supported in $\Y_0$.
Note that the action of $G$ on $N^1(X\times \Aff^1)$ is induced by the action of $G$ on $N^1(X)$. Since the Mori cone $NE(X)$ is polyhedral, the action $G$ is finite on $N^1(X)$,  and hence is finite on $N^1(X\times \Aff^1)$. The action of $G$ on $\Y_0$ permutes irreducible components of $\Y_0$, so it is also finite on N'. Therefore, we get the conclusion.

Now since $\phi^\ast: N^1(\X)\to N^1(\Y)$ is an injection, the action of $G$ on $N^1(\X)$ also factors through a finite group.
\end{proof}

In order to state the equivariant version of Li-Xu's thoerem, We also recall the definition of $G$-irreducibility.
\begin{definition}
A scheme $X$ with a group action $G$ is called $G$-irreducible if the underlying topological space of $X$ cannot be written as the union of $X_1$ and $X_2$ with $X_1$ and $X_2$ two $G$-invariant proper closed subsets.
\end{definition}
Note that equivalently, $X$ is $G$-irreducible if $G$ acts transitively on the irreducible components of $X$.

\begin{theorem}\label{GS}
For any G-equivariant normal test configuration $(\X,\LL)/\Aff^1$ of $(X,-K_X)$, there exists a finite morphism $\phi:\Aff^1 \to \Aff^1$, a test configuration $(\X^s,-K_{\X^s})$ with the central fiber being reduced and $G$-irreducible and a both $\C^\ast$- and $G$-equivariant birational map  $\X^s\dashrightarrow \X\times_\phi \Aff^1$ over $\Aff^1$, such that for any $0\leq\delta\leq 1$, we have
$$
\DF(\X^s,-K_{\X^s})-\delta J^{\textrm{NA}}(\X^s,-K_{\X^s}) \leq \deg \phi \left( \DF(\X,\LL)-\delta J^{\textrm{NA}}(\X,\LL)\right).
$$
\end{theorem}
\begin{proof}
By Lemma \ref{GMMPtc}, we can run a $G$-equivariant version of each MMP step in the proof of the main theorem in \cite{LX14}. Then we get a $G$-equivariant semi-test configuration $(\X^m,-K_{\X^m})$ such that the central fiber $\X^m_0$ is reduced and $G$-irreducible, and~$-K_{\X^m}$ is relatively semiample. Let $\X^s=\proj_{\Aff^1} R(\X^m,-K_{\X^m})$ be the relative anticanonical model. This yields the $G$-equivariant test configuration $(\X^s,-K_{\X^s})$. The computation in \cite{LX14} and \cite{Fuj16} gives us the inequality in terms of the Donaldson-Futaki invariant. 
\end{proof}
Note that according to the MMP steps in \cite{LX14}, we know that $(\X^m,\X_0^m)$ is dlt. However, after we take the relative anti-canonical model, the pair $(\X^s,\X_0^s)$ is only log canonical and not necessarily dlt any more. 
\begin{definition}
A $G$-equivariant normal test configuration $(\X,\LL)$ is called a weakly $G$-special test configuration if the central fiber $\X_0$ is reduced and $G$-irreducible.
\end{definition}
For a weakly $G$-special test configuration $(\X,\LL)$, since both $\LL$ and $K_{\X}$ are $G$-invariant, we know that $\LL+K_{\X}$ is a $G$-invariant divisor supported on the central fiber $\X_0$ which is $G$-irreducible. Therefore it can only be a multiple of the whole fiber $\X_0$. By definition of the Donaldson-Futaki invariant, we have $\DF(\X,\LL)=\DF(\X,-K_{\X})$ and $J^{\textrm{NA}}(\X,\LL)=J^{\textrm{NA}}(\X,-K_{\X})$. As in the usual K-stability case, we know from Theorem \ref{GS} that it is enough to check only weakly $G$-special test configurations for $G$-equivariant K-stability. 
\begin{remark}
Note that by the computation in \cite[Section 3]{Fuj16}, Theorem \ref{GS} also holds if we replace the Donaldson Futaki invariants with Ding invariants (see \cite[Definition 3.1]{Fuj15} for the definitoon of Ding invariant). Since these two invariants are the same for weakly $G$-special test configurations by \cite[Theorem 3.2]{Fuj15}, we know that we can also define various notions of K-stability using the Ding invariant $\Ding(\X,\LL)$ in place of the Donaldson Futaki invariant $\DF(\X,\LL)$ in Definition \ref{def:k-stab}. We will use this fact in the proof of Theorem \ref{part1}.
\end{remark}

\subsection{Filtrations and test configurations}
A test configuration $(\X,\LL)$ of $(X,-K_X)$ induces a filtration $\F$ on $V_k=H^0(X,-kK_X)$ in the following way:
$$
\F^xV_k=\{s\in V_k|t^{-\lceil x\rceil}\bar{s}\in H^0(\X,k\LL)\},
$$
where $\bar{s}$ is the $\C^\ast$-invariant section of $k\LL$ on $\X\backslash \X_0$ induced by $s$. Note that $\F$ is decreasing, left-continuous, multiplicative and linearly bounded (see \cite[Section 1.1]{BHJ17}  and \cite[Definition 4.1]{Fuj15} for definitions). Filtrations in this paper will always be assumed to satisfy these four properties.

Conversely, let $\F$ be a filtration on $V_\bullet$ such that $\bigoplus_{k\in \Z_{\geq 0},j\in \Z} \F^jV_k$ is finitely generated. We may assume it is generated in degree $k=1$. Then we can define a test configuration
$$
(\X,\LL)=\left(\proj_{\Aff^1}\bigoplus_{k\in \Z_{\geq 0},j\in \Z} t^{-j}\F^j V_k,\OS(1)\right),
$$ 
where the $\proj$ is taken with respect to the grading of $k$, and the $\Z$-grading of $j$ determines an equivariant $\C^\ast$-action on the test configuration over $\Aff^1=\spec\C[t]$. Note that the central fiber of the test configuration is 
$$
\X_0=\proj \bigoplus_{k\in \Z_{\geq 0},j\in \Z} \gr_{\F}^jV_k,
$$
where
$$
\gr_{\F}^jV_k=\frac{{\F}^jV_k}{{\F}^{j+1}V_k}
$$
is the graded piece of the filtration $\F^\bullet$. The induced $\C^\ast$-action from the test configuration on $\X_0$ is also determined by the $\Z$-grading of $j$. 
The following proposition gives the relation between filtrations and test configurations.
\begin{proposition}[Proposition 2.15, \cite{BHJ17}]\label{BHJ}
The above construction sets up a one-to-one correspondence between test configurations of $(X,-K_X)$ and finitely generated filtrations on $V_\bullet$.
\end{proposition}

For any prime divisor $F$ over $X$, we can construct a $G$-invariant filtration
\begin{equation}\label{filt}
\F^x V_r=
\begin{cases}
H^0\left(X,\OS_X(-rK_X)\otimes\aid_{\lceil x\rceil}(G\cdot \ord_F)\right),&~x\geq 0,\\
V_r,&~x	<0.
\end{cases}
\end{equation}
which induces a $G$-equivariant test configuration.

To conclude this section, we look at some basic examples that illustrate the difference between $G$-equivariant K-stability and usual K-stability.
\begin{example}
Consider the projective space $X=\PP^n$ with $G=\PGL(n+1)$-action. A nontrivial $\C^\ast$-action on $\PP^n$ does not commute with $\PGL(n+1)$. Then the only $G$-equivariant test configuration of $(\PP^n,-K_{\PP^n})$ is the trivial test configuration. By definition we know that $\PP^n$ is uniformly $G$-equivariantly K-stable. Note that for any $G$-pseudovaluation $G\cdot v$, we have that $\aid_x(G\cdot v)=(0)$ for any $x>0$. Therefore for any prime divisor $F$ over $\PP^n$, we know that the corresponding $G$-invariant filtration is
$$
\F^x V_r=
\begin{cases}
0,&~x> 0,\\
V_r,&~x \leq 0,
\end{cases}
$$
which induces exactly the trivial test configuration.
\end{example} 

\begin{example}\label{csgo}
Consider $X=\PP^1\times\PP^1$ with $G=\PGL(2)$ acting on the first component. Pick any point $p\in X$. Let $E$ be the exceptional divisor of the blow-up of $X$ at $p$. Let $H$ be the horizontal line through $p$, and we know that $H$ is the orbit of $p$ under $G$-action. Therefore $E$ and $H$ induce the same $G$-invariant filtration. Note that $E$ is not of finite orbit, but $H$ is $G$-invariant. The compactified test configuration corresponding to the $G$-invariant filtration is $\pi:~\PP^1\times \FF_1\to \PP^1$, with $G$ acting on the first component and $\pi$ induced by the Hirzebruch surface $\FF_1\to \PP^1$. 

Similar examples can also be constructed in the same way for other group action such as a torus action $(\C^\ast)^r$ on $\PP^n$.
\end{example}

\begin{example}
Cheltsov and Shramov studied a special class $V^\ast_{22}$ of Fano threefolds of degree 22 admitting $(\C^\ast\rtimes \Z/2)$-action. Denote $G=\C^\ast\rtimes \Z/2$. They computed in \cite{CS18} the equivariant alpha invariant $\alpha_G(X)$ for every smooth Fano threefold of type $V^\ast_{22}$, and showed that all but two of them have $\alpha_G(X)=4/5$ (see Section \ref{sec:tian} for the definition of equivariant alpha invariant). Then by \cite[Theorem 1.10]{Al} (see also Theorem \ref{eqal} in Section \ref{sec:tian}), we know that these Fano threefolds with $\alpha_G(X)=4/5$ are uniformly $G$-equivariantly K-stable. However, since the automorphism groups contain $\C^\ast$ and hence are not discrete, they are not uniformly K-stable in the usual sense.
\end{example}

\section{Equivariant valuative criteria}
We separate the proof of Theorem \ref{val} into 3 parts. We first prove the following theorem which gives a necessary valuative condition of uniform equivariant K-stability in Theorem \ref{val}.
\begin{theorem}\label{part1}
Let $X$ be a $\Q$-Fano variety with $G\subset \aut(X)$ a group action on $X$. If $X$ is uniformly $G$-equivariantly K-stable, then there exists $0<\delta<1$, such that $\beta^G(F)\geq \delta j^G(F)$ for any finite-orbit prime divisor $F$ over $X$.
\end{theorem}
\begin{proof}
We follow the idea of the proof of Theorem 4.1 in \cite{Fuj16}. For simplicity, We may assume $-K_X$ is already Cartier. Given any prime divisor $F$ of finite orbit, let $\pi:Y\to X$ be a $G$-equivariant resolution such that $F$ is a smooth divisor on $Y$. Following the notation in \eqref{filt}, we consider the $G$-invariant filtration of $\F^x V_r$ defined by $F$. Note that $\F$ is saturated.
Let
$
I_{(r,x)}:=\im (\F^x V_r\otimes\OS_X(rK_X)\to \OS_X)
$
be the base ideal of $\F^x V_r$. Suppose $F_1,\ldots,F_N$ form the orbit of $F$ under the $G$-action with $F_1=F$. We have 
$$I_{(r,x)}\cdot \OS_Y\subset \OS_Y\left(-\lceil x\rceil \sum_{i=1}^N F_i\right).
$$
Fix a sufficiently large integer $e_+>\tau^G(F)$ and sufficiently small $e_-<0$. We set the flag ideal $\I_r\subset \OS_{X\times \Aff^1}$ to be
$$
\I_r:=I_{(r,re_+)}+I_{(r,re_+-1)}t+\cdots+I_{(r,re_-+1)}t^{r(e_+-e_-)-1}+\left(t^{r(e_+-e_-)}\right).
$$
Let $\Pi_r:\X_r\to X\times \Aff^1$ be the blow-up along $\I_r$, and $E_r$ the exceptional divisor. Set $\LL_r:=\Pi_r^\ast(-K_{X\times\Aff^1})-\frac{1}{r}E_r$. Note that all $I_{(r,x)}$'s are $G$-invariant, so $\I_r$ is a $G$-invariant ideal, and we get that $(\X_r,\LL_r)$ is a $G$-equivariant semi-test configuration. Since $X$ is uniformly $G$-equivariantly K-stable, we know that $\Ding(\X_r,\LL_r)\geq \delta J^{NA}(\X_r,\LL_r)$ for some $\delta\in (0,1)$ independent of $r$.

Note that 
$$
\I_r\cdot \OS_{Y\times \Aff^1}\subset \left(\OS_Y\left(-\sum_{i=1}^N F_i\right)+(t)\right)^{re_+}.
$$
Now if we follow the same computation of $\Ding(\X_r,\LL_r)$  and take $r\to \infty$ as in the proof of Theorem 4.1 in \cite{Fuj16}, we will get $\beta^G(F)\geq \delta j^G(F)$.
\end{proof}

\begin{remark}
When $F$ is not of finite orbit, it is in general not possible to find a $G$-equivariant resolution $Y\to X$ as in the above proof. Therefore the assumption that $F$ is of finite orbit is necessary for the proof. For a finite-orbit prime divisor $F$, we can write the equivariant beta invariant of $F$ in the following way
$$
\beta^G(F):=A_X(F)(-K_X)^n-\int_0^{+\infty}\vol_Y\left(\pi^\ast (-K_X)-\sum_{i=1}^n F_i\right)\,dx.
$$
where $\pi:Y\to X$ is a $G$-equivariant resolution of $X$, and $F_i$'s form the orbit of $F$ as in the above proof. It is crucial in the computation that the sum $\sum F_i$ we consider here is a reduced divisor. If we instead  take the sum of all $gF$ for $g\in G$, it can be a nonreduced divisor. 
\end{remark}

Next we study the relation between Donaldson-Futaki invariants of weakly $G$-special test configurations and equivariant beta invariants.

\begin{theorem}\label{tF}
Let $(\X,\LL)$ be a nontrivial weakly $G$-special test configuration. Suppose the central fiber of $\X$ can be decomposed into irreducible components $\X_0^1,\ldots,\X_0^N$. Let $v_i$ be the restriction on $X$ of the divisorial valuation $\ord_{\X_0^{i}}$. Then each $v_i$ is a divisorial valuation and is $G$-dreamy. Moreover, we have $\DF(\X,\LL)=\beta^G(v_i)/(-K_X)^n$ and $J^{\textrm{NA}}(\X,\LL)=j^G(v_i)/(-K_X)^n$ for any $i$.
\end{theorem}
\begin{proof}
First note that we may assume $\LL=-K_{\X/\Aff^1}$. Next we claim that each $v_i$ is a divisorial valuation on $X$ corresponding to distinct divisor $F_i$ over $X$. Indeed, Suppose $\ord_{\X_0^{i}}$ and $\ord_{\X_0^{j}}$ restrict to the same valuation $v_i$ and $v_j$ on $X$. Then since  $\ord_{\X_0^{i}}(t)=\ord_{\X_0^{j}}(t)=1$, we know that the two valuations are the same on $\X$. Lemma 4.1 in \cite{BHJ17} shows that $v_i$'s are divisorial valuations on $X$. Denote $v_i=c\ord_{F_i}$, where $F_i$ is a prime divisor on $X$. Since the $G$-action permutes all the irreducible components of the central fiber $\X_0$, we know that ${F_1,\ldots,F_N}$ form the orbit of $F_1$ under the $G$-action. 

Now consider a common resolution of $X\times \Aff^1$ and $\X$ which gives the following diagram over $\Aff^1$: 
\begin{center}
\begin{tikzcd}
& \Y \arrow[dl,"p"']
	\arrow[dr,"q"] & \\
X\times \Aff^1 \arrow [rr,dotted] & & \X.
\end{tikzcd}
\end{center}
Let $V_k=H^0(X,-kK_X)$ and $\F_{\Y}^xV_k$ be the filtration on $V_\bullet$ induced by $(\Y,q^\ast \LL)$. Following the same argument as in the proof of Theorem 5.1 in \cite{Fuj16}, we have
$$
\F_{\Y}^xV_k=\{f\in V_k| v_i(f)\geq kA_X(F_i)+x,~i=1,\ldots,N\}.
$$
Note that  $A_X(F_i)=A_X(F_1)$ for any $i$. Therefore we can write the filtration $\F_{\Y}^xV_k$ as
$$
\F_{\Y}^xV_k=
\begin{cases}
H^0\left(X,-kK_X-(kA_X(F_1)+x)\sum_{i=1}^NF_i\right) & x\geq -kA_X(F_1),\\
V_k & x< -kA_X(F_1).
\end{cases}
$$
Finite generation of $\bigoplus_{j,k}\F_{\Y}^jV_k$ would imply finite generation of 
$$
\bigoplus_{j,k\geq 0}H^0\left(X,-kK_X-j\sum_{i=1}^NF_i\right)
$$ 
as in the proof of Theorem 5.1 in \cite{Fuj16}, and hence $F_i$ is $G$-dreamy. The same computation in that proof would also give us $\DF(\X,\LL)=\beta^G(v_i)/(-K_X)^n$ and $J^{\textrm{NA}}(\X,\LL)=j^G(v_i)/(-K_X)^n$ for any $i$.
\end{proof}
The following theorem is an immediate consequence of Theorem \ref{tF}.
\begin{theorem}\label{part2}
If there exists some $0<\delta<1$, such that $\beta^G(F)>0(\geq \delta j^G(F))$ for any $G$-dreamy divisor $F$ over $X$, then $X$ is (uniformly) $G$-equivariantly K-stable.
\end{theorem}

By Proposition \ref{BHJ}, we have a one-to-one correspondence between finitely generated filtrations on $V_\bullet$ and test configurations. Then  combining Proposition \ref{BHJ} with Theorem \ref{tF}, we have the following theorem: 
\begin{theorem}
Let $X$ be a $\Q$-Fano variety and $F$ a $G$-dreamy divisor over $X$. Define a filtration $\F$ on $V_\bullet$ as in \eqref{filt}. Then the test configuration 
$$(\X,\LL)=\left(\proj_{\Aff^1}\bigoplus_{k\in \Z_{\geq 0},j\in \Z} t^{-j}\F^j V_k,\OS(1)\right)$$
is a weakly $G$-special test configuration. Moreover, we have $\DF(\X,\LL)=\beta^G(F)/(-K_X)^n$ and $J^{\textrm{NA}}(\X,\LL)=j^G(F)/(-K_X)^n$.
\end{theorem}
\begin{proof}
We only need to show that $\X_0$ is reduced and $G$-irreducible. Note that
$$
\X_0=\proj \bigoplus_{k,j\geq 0} \gr_{\F}^jV_k.
$$
Pick any $f\in {\F}^{j}V_{k}\backslash {\F}^{j+1}V_{k}$. We have $G\cdot \ord_F(f)=j$. For any positive integer $l$, we have $G\cdot \ord_F(f^l)=lj$ (see also \cite[Remark 3.12]{arcs}), or equivalently $f^l\in {\F}^{lj}V_{k}\backslash {\F}^{lj+1}V_{k}$. Therefore 
we know that $\bigoplus_{k,j\geq 0} \gr_{\F}^jV_k$ is reduced. 

Now pick any $f_i\in {\F}^{j_i}V_{k_i}\backslash {\F}^{j_i+1}V_{k_i}$ for $i=1,2$. Let $F_1,\ldots,F_N$ form the orbit of $F$ under $G$-action. Since the orbit is finite, we can find $F_{l_i}$ such that $\ord_{F_{l_i}}(f_i)=G\cdot \ord_F(f_i)=j_i$ for $i=1,2$. Suppose $\ord_{F_{l_2}}=g\cdot\ord_{F_{l_1}}$ for some $g\in G$. Then we have
$$
G\cdot \ord_F (g(f_1)f_2)=\ord_{F_{l_2}}(g(f_1)f_2)=j_1+j_2,
$$
and consequently $g(f_1)f_2\in {\F}^{j_1+j_2}V_{k_1+k_2}\backslash {\F}^{j_1+j_2+1}V_{k_1+k_2}$. Therefore we know that $\X_0$ is $G$-irreducible.
\end{proof}

An immediate consequence is the following theorem:
\begin{theorem}\label{part3}
If $X$ is $G$-equivariantly K-stable, then $\beta^G(F)> 0$ for any $G$-dreamy divisor $F$ over $X$. 
\end{theorem}

Finally we finish the proof of Theorem \ref{val} by combining the above results. Theorem \ref{part1} and Theorem \ref{part2} gives the valuative criterion for uniform equivariant K-stability, which is part (1) in Theorem \ref{val}. If we set $\delta=0$ in Theorem \ref{part1} and Theorem \ref{part2}, we get the corresponding valuative criterion for equivariant K-semistability, which is part (2) in Theorem \ref{val}. Theorem \ref{part2} and Theorem \ref{part3} gives us the valuative criterion for equivariant K-stability, which is part (3) in Theorem \ref{val}.

\section{Equivariant normalized volumes}
Let $X$ be an $n$-dimensional $\Q$-Fano variety with group action $G$ and $F$ a prime divisor over $X$. Denote by $Y=C(X,-K_X)$ the cone over $X$ with respect to the polarization $-K_X$ and $o\in Y$ the vertex of the cone. Suppose $\pi:~Z=Bl_oY\to Y$ is the blow-up of $Y$ at $o$. Let $E$ be the exceptional divisor, and $\FF$ the pull back of $F$ to Z. Denote the divisorial valuation $\ord_E$ by $v_0$ and $\ord_{\FF}$ by $v_F$. Then for $t\geq 0$, let $v_t$ be the quasi-monomial valuation between $v_0$ and $v_F$ with weight $(1,t)$. For $f_m\in V_m=H^0(X,-mK_X)$, we have $v_t(f_m)=v_0(f_m)+tv_F(f_m)=m+t\ord_F(f_m)$.  Note that $v_t$ is still centered at $o$. Since there is a natural $G$-action induced on the cone $Y$ and the blow-up $Z$, we consider the $G$-pseudovaluation $G\cdot v_t$. The following proposition gives a relation between the derivative of the normalized volume $\widehat{\vol}(G\cdot v_t)$ and $\beta^G(F)$.
\begin{proposition}\label{der}
Under the above notations, we have
$$
\frac{d}{dt}\widehat{\vol}(v_t)\bigg|_{t=0}=(n+1)\beta^G(F).
$$
\end{proposition}
\begin{proof}
The idea of the proof basically follows from \cite{Li17K} and \cite{LX16}.
First of all, we have $A_Y(\FF)=A_X(F)$, and $A_Y(v_0)=1$. Therefore $A_Y(v_t)=1+tA_X(F)$. Next we compute the volume of $G\cdot v_t$. Let $V_m=H^0(X,-mK_X)$ and $V=\oplus V_m$. Note that for $f\in V_m$, we have $v_0(f)=m$. Then
$$
\dim V/\aid_\lambda(G\cdot v_t)=\sum_{m=0}^{\lfloor \lambda\rfloor} \dim V_m/\aid_\lambda(G\cdot v_t)=\sum_{m=0}^{\lfloor \lambda\rfloor} \dim V_m-\sum_{m=0}^{\lfloor \lambda\rfloor} \dim V_m\cap \aid_\lambda(G\cdot v_t).
$$
By asymptotic Riemann-Roch, we know that
$$
\sum_{m=0}^{\lfloor \lambda\rfloor} \dim V_m=\frac{(-K_X)^n\lambda^{n+1}}{(n+1)!}+O(\lambda^n).
$$
On the other hand, for any $f\in V_m$, we know that $g\cdot v_t(f)\geq \lambda$ is equivalent as $g\cdot v_F(f)\geq \frac{\lambda-m}{t}$. Then we know that
$$
V_m\cap \aid_\lambda(G\cdot v_t)=H^0\left(\OS_X(-mK_X)\otimes\aid_{\frac{\lambda-m}{t}}(G\cdot \ord_F) \right).
$$
According to Lemma 4.5 in \cite{Li17K}, we know that
\begin{align*}
&\sum_{m=0}^{\lfloor \lambda\rfloor} \dim V_m\cap \aid_\lambda(G\cdot v_t)\\
=&
\frac{\lambda^{n+1}}{n!}\int_0^{+\infty}\vol_X\left(\OS_X(-K_X)\otimes\aid_x(G\cdot \ord_F)\right)\frac{t}{(1+tx)^{n+2}}\,dx+O(\lambda^n).
\end{align*}
Putting the above expressions together, we have
$$
\vol(G\cdot v_t)=(-K_X)^n-(n+1)\int_0^{+\infty}\vol_X\left(\OS_X(-K_X)\otimes\aid_x(G\cdot \ord_F)\right)\frac{t}{(1+tx)^{n+2}}\,dx.
$$
Taking the derivative we get
$$
\frac{d}{dt}\widehat{\vol}(G\cdot v_t)\bigg|_{t=0}=(n+1)\beta^G(F).
$$
\end{proof}
An immediate consequence of Proposition \ref{der} gives one direction of Theorem \ref{norm}:
\begin{corollary}
If the normalized volume function $\widehat{\vol}$ is minimized at $v_0$ among all finite-orbit $G$-pseudovaluations on $Y$ centered at $o$, then $X$ is $G$-equivariantly K-semistable.
\end{corollary}

Repeating a similar computation together with convexity of the normalized volume function of $G\cdot v_t$ with respect to $t$ as in the proof of Theorem 4.5 in \cite{LX16} also gives the other direction of Theorem \ref{norm}.

\section{Equivariant alpha invariant and Tian's criterion}\label{sec:tian}
Let $X$ be a variety with $G\subset \aut(X)$ a group action on $X$. By replacing usual valuations with $G$-pseudovaluations, we can define the $G$-log canonical threshold of any effective divisor $D$ to be
$$
\glct(D):=\inf_{E}\frac{A_X(E)}{G\cdot \ord_E(D)},
$$
where $E$ runs through all prime divisors over $X$.
Next assume in addition that $X$ is $\Q$-Fano. We define the $G$-equivariant alpha invariant of $X$ to be
$$
\alpha_G(X)=\inf\{\glct(D)|0\leq D\sim_\Q -K_X\}.
$$
\begin{remark}
Note that Tian first defines $\alpha_G(X)$ analytically in \cite{tian_alpha}. It is then shown in the appendix of \cite{analpha} that the analytic definition of $\alpha_G(X)$ is the same as the following algebraic one:
$$
\alpha_G(X)=\inf_m\left\{\lct\left(X,\frac{1}{m}\Sigma\right)\bigg|\Sigma~\textrm{is a}~G\textrm{-invariant linear subsystem in}~|-mK_X|\right\}.
$$
The above two algebraic definitions of $\alpha_G(X)$ are in fact the same. Indeed, for any $G$-invariant linear system $\Sigma\sim -mK_X$, pick any divisor $D\in \Sigma$. Then for any prime divisor $E$ over $X$ and $g\in G$, we have $\ord_E(gD)\geq \ord_E(\Sigma)$. Therefore we know that 
$$
\glct\left(\frac{1}{m}D\right)\leq \lct\left(\frac{1}{m}\Sigma\right).
$$
Conversely, for any effective divisor $D\in |-mK_X|$, let $\Sigma$ be the linear subsystem of $|-mK_X|$ spanned by $\{gD|g\in G\}$. Then for any effective divisor $D'\in \Sigma$, we know that $\ord_E(D')\geq G\cdot \ord_E(D)$ and hence $\ord_E(\Sigma)\geq G\cdot \ord_E(D)$. Therefore we know that 
$$
\glct\left(\frac{1}{m}D\right)\geq \lct\left(\frac{1}{m}\Sigma\right).
$$
\end{remark}

Next we will give another proof of the following result in \cite{Al} which is the $G$-equivariant version of Tian's criterion.
\begin{theorem}[Theorem 1.10, \cite{Al}]\label{eqal}
Let $X$ be a $\Q$-Fano variety of dimension $n$ and $G\subset \aut(X)$ a group action on $X$. If  $\alpha_G(X)>n/(n+1)$ (resp. $\geq n/(n+1)$), then $X$ is uniformly $G$-equivariantly K-stable (resp. $G$-equivariantly K-semistable).
\end{theorem}
Before we prove Theorem \ref{eqal}, we would like to introduce for simplicity the following notations related to equivariant beta invariants of finite-orbit prime divisors. For any prime divisor $F$ over $X$ of finite orbit under the $G$-action, let $\pi:~Y\to X$ be a $G$-equivariant birational morphism such that $F_1,\ldots,F_N$ are prime divisors on $Y$ forming the orbit of $F$ under the $G$-action with $F_1=F$. We define
$$
S^G(F)=\frac{1}{(-K_X)^n}\int_0^{+\infty}\vol_Y\left(\pi^\ast(-K_X)-x\sum F_i \right)\,dx.
$$
Then we can write $\beta^G(F)$ as
$$
\beta^G(F)=(-K_X)^n\left(A_X(F)-S^G(F)\right).
$$
Also note that under the above notation, we have
$$
\tau^G(F)=\sup\left\{t>0\left|\vol_Y\left(\pi^\ast(-K_X)-t\sum F_i \right)>0\right\}\right.
$$ 
and
$$
j^G(F)=(-K_X)^n\left(\tau^G(F)-S^G(F)\right).
$$
\begin{proof}[Proof of Theorem \ref{eqal}]
The idea of the proof follows from \cite{fujita_2017}. We only prove uniform $G$-equivariant K-staiblity of $X$ since the proof for $G$-equivariant K-semistability is almost identical.

Take any weakly $G$-special test configuration $(\X,\LL)$. Let $F$ be the prime divisor over $X$ corresponding to the divisorial valuation on $X$ induced by one of the irreducible components of $\X_0$. Then the orbit of $F$ under the $G$-action is induced by all irreducible components of $\X_0$ and hence is finite. Let $F_1,\ldots,F_N$ form the orbit of $F$ with $F_1=F$. Under the above notation, it suffices to show that 
$$
A_X(F)-S^G(F)\geq\delta\frac{j^G(F)}{(-K_X)^n}
$$
for some $\delta\in (0,1)$ independent of $F$.

Using integration by parts, we have
$$
\int_0^{\tau^G(F)}\left(x-S^G(F)\right)\frac{d}{dx}\vol_Y\left(-\pi^\ast K_X-x\sum F_i \right)\,dx=0.
$$
Note that by Theorem A and Theorem B of \cite{diff}, we have 
$$
-\frac{1}{n}\frac{d}{dx}\vol_Y\left(-\pi^\ast K_X-x\sum F_i \right)=N\vol_{Y|F}\left(-\pi^\ast K_X-x\sum F_i \right).
$$
where $\vol_{Y|F}$ denotes the restricted volume (see \cite{ELMNP} for definition). For simplicity, we use $V(x)$ to denote the restricted volume function $\vol_{Y|F}\left(-\pi^\ast K_X-x\sum F_i \right)$. Then we have 
$$
\int_0^{\tau^G(F)}\left(x-S^G(F)\right)V(x)\,dx=0.
$$
Using log concavity of restricted volume, we have
$$
\left(x-S^G(F)\right)V(x)\leq \left(x-S^G(F)\right)V\left(S^G(F)\right) \left(\frac{x}{S^G(F)}\right)^{n-1}.
$$
Therefore we get that 
\begin{equation}\label{eqj}
S^G(F)\leq \frac{n}{n+1}\tau^G(F).
\end{equation}

Now suppose that
$$
A_X(F)-S^G(F)<\delta\frac{j^G(F)}{(-K_X)^n}
$$ 
for any $\delta\in (0,1)$. Recall that 
$$
\frac{j^G(F)}{(-K_X)^n}=\left(\tau^G(F)-S^G(F)\right),
$$
so we have
$$
A_X(F)<\delta\tau^G(F)+(1-\delta)S^G(F).
$$
Combining the above inequality with \eqref{eqj}, we get
$$
A_X(F)<\left(\frac{1}{n+1}\delta+\frac{n}{n+1}\right)\tau^G(F).
$$
Let $\delta\to 0$, and we know that 
$$
A_X(F)\leq \frac{n}{n+1}\tau^G(F).
$$
For arbitrarily small $\epsilon>0$, pick $0\leq D\sim_\Q -K_X$ such that $G\cdot \ord_F(D)=\tau^G(F)-\epsilon$.
Then we know that
$$
\glct(D)\leq \frac{A_X(F)}{G\cdot \ord_F(D)}\leq \frac{n}{n+1}\frac{\tau^G(F)}{\tau^G(F)-\epsilon},
$$
which implies $\alpha_G(X)\leq n/(n+1)$, contradicting to the assumption that $\alpha_G(X)>n/(n+1)$.
\end{proof}

\begin{remark}
Note that as a consequence of the conjecture that for a reductive group $G$, $X$ is $G$-equivariantly K-semistable (resp. $G$-equivariantly K-polystable) if and only if $X$ is K-semistable (resp. K-polystable), we have that if $\alpha_G(X)\geq n/(n+1)$ (resp.~$\alpha_G(X)>n/(n+1)$), then $X$ is K-semistable (resp. K-polystable). This generalizes the original Tian's criterion to any $\Q$-Fano varieties.
\end{remark}
As mentioned in the introduction, the above conjecture is proved in \cite{zhuang2020} after the original version of the paper is posted online. Therefore, we now indeed have the generalized Tian's criterion for $\Q$-Fano varieties.

\bibliographystyle{siam}
\bibliography{ref}

\begin{thebibliography}{10}

\bibitem{diff}
{\sc S.~Boucksom, C.~Favre, and M.~Jonsson}, {\em Differentiability of volumes
  of divisors and a problem of {T}eissier}, J. Algebraic Geom., 18 (2009),
  pp.~279--308.

\bibitem{BHJ17}
{\sc S.~Boucksom, T.~Hisamoto, and M.~Jonsson}, {\em Uniform {K}-stability,
  {D}uistermaat-{H}eckman measures and singularities of pairs}, Ann. Inst.
  Fourier (Grenoble), 67 (2017), pp.~743--841.

\bibitem{CS18}
{\sc I.~Cheltsov and C.~Shramov}, {\em K{\"a}hler-{E}instein {F}ano threefolds
  of degree 22},  (2018).
\newblock \href{https://arxiv.org/abs/1803.02774}{\textsf{arXiv:1803.02774}}.

\bibitem{analpha}
{\sc I.~A. Cheltsov and K.~A. Shramov}, {\em Log-canonical thresholds for
  nonsingular {F}ano threefolds, with an appendix by {J}.{P}. {D}emailly},
  Uspekhi Mat. Nauk, 63 (2008), pp.~73--180.

\bibitem{DS16}
{\sc V.~Datar and G.~Sz\'{e}kelyhidi}, {\em K\"{a}hler-{E}instein metrics along
  the smooth continuity method}, Geom. Funct. Anal., 26 (2016), pp.~975--1010.

\bibitem{arcs}
{\sc T.~de~Fernex and M.~Musta\c{t}\u{a}}, {\em The volume of a set of arcs on
  a variety}, Rev. Roumaine Math. Pures Appl., 60 (2015), pp.~375--401.

\bibitem{ELMNP}
{\sc L.~Ein, R.~Lazarsfeld, M.~Musta\c{t}\u{a}, M.~Nakamaye, and M.~Popa}, {\em
  Restricted volumes and base loci of linear series}, Amer. J. Math., 131
  (2009), pp.~607--651.

\bibitem{fujita_2017}
{\sc K.~Fujita}, {\em K-stability of {F}ano manifolds with not small alpha
  invariants}, Journal of the Institute of Mathematics of Jussieu,  (2017),
  p.~1–12.

\bibitem{Fuj15}
\leavevmode\vrule height 2pt depth -1.6pt width 23pt, {\em Optimal bounds for
  the volumes of {K}\"{a}hler-{E}instein {F}ano manifolds}, Amer. J. Math., 140
  (2018), pp.~391--414.

\bibitem{Fuj16}
\leavevmode\vrule height 2pt depth -1.6pt width 23pt, {\em A valuative
  criterion for uniform {K}-stability of {${\mathbb Q}$}-{F}ano varieties}, J.
  Reine Angew. Math., 751 (2019), pp.~309--338.

\bibitem{golota2019delta}
{\sc A.~Golota}, {\em Delta-invariants for {F}ano varieties with large
  automorphism groups},  (2019).
\newblock \href{https://arxiv.org/abs/1907.06261}{\textsf{arXiv:1907.06261}}.

\bibitem{KM98}
{\sc J.~Koll\'{a}r and S.~Mori}, {\em Birational geometry of algebraic
  varieties}, vol.~134 of Cambridge Tracts in Mathematics, Cambridge University
  Press, Cambridge, 1998.
\newblock With the collaboration of C. H. Clemens and A. Corti, Translated from
  the 1998 Japanese original.

\bibitem{Li17K}
{\sc C.~Li}, {\em K-semistability is equivariant volume minimization}, Duke
  Math. J., 166 (2017), pp.~3147--3218.

\bibitem{Li18}
\leavevmode\vrule height 2pt depth -1.6pt width 23pt, {\em Minimizing
  normalized volumes of valuations}, Math. Z., 289 (2018), pp.~491--513.

\bibitem{LL19}
{\sc C.~Li and Y.~Liu}, {\em K\"{a}hler-{E}instein metrics and volume
  minimization}, Adv. Math., 341 (2019), pp.~440--492.

\bibitem{LWX18}
{\sc C.~Li, X.~Wang, and C.~Xu}, {\em Algebraicity of the metric tangent cones
  and equivariant {K}-stability},  (2018).
\newblock \href{https://arxiv.org/abs/1805.03393}{\textsf{arXiv:1805.03393}}.

\bibitem{LX14}
{\sc C.~Li and C.~Xu}, {\em Special test configuration and {K}-stability of
  {F}ano varieties}, Ann. of Math. (2), 180 (2014), pp.~197--232.

\bibitem{LX16}
\leavevmode\vrule height 2pt depth -1.6pt width 23pt, {\em Stability of
  valuations and {K}oll\'{a}r components}, J. Eur. Math. Soc. (JEMS), to
  appear,  (2016).
\newblock \href{https://arxiv.org/abs/1604.05398}{\textsf{arXiv:1604.05398}}.

\bibitem{LZ20}
{\sc Y.~Liu and Z.~Zhu}, {\em Equivariant {K}-stability under finite group
  action},  (2020).
\newblock \href{https://arxiv.org/abs/2001.10557}{\textsf{arXiv:2001.10557}}.

\bibitem{Oda13a}
{\sc Y.~Odaka}, {\em A generalization of the {R}oss-{T}homas slope theory},
  Osaka J. Math., 50 (2013), pp.~171--185.

\bibitem{Al}
{\sc Y.~Odaka and Y.~Sano}, {\em Alpha invariant and {K}-stability of {$\mathbb
  Q$}-{F}ano varieties}, Adv. Math., 229 (2012), pp.~2818--2834.

\bibitem{tian_alpha}
{\sc G.~Tian}, {\em On {K}\"ahler-{E}instein metrics on certain {K}\"ahler
  manifolds with {$C_1(M)>0$}}, Invent. Math., 89 (1987), pp.~225--246.

\bibitem{XZ20}
{\sc C.~Xu and Z.~Zhuang}, {\em Uniqueness of the minimizer of the normalized
  volume function},  (2020).
\newblock \href{https://arxiv.org/abs/2005.08303}{\textsf{arXiv:2005.08303}}.

\bibitem{zhuang2020}
{\sc Z.~Zhuang}, {\em Optimal destabilizing centers and equivariant
  {K}-stability},  (2020).
\newblock \href{https://arxiv.org/abs/2004.09413}{\textsf{arXiv:2004.09413}}.

\end{thebibliography}
\end{document}